\documentclass[graybox]{svmult}

\usepackage{mathptmx}       
\usepackage{helvet}         
\usepackage{courier}        
\usepackage{makeidx}         
\usepackage{graphicx}        
\usepackage{multicol}        

\begin{document}

\title*{On a stochastic Ricker competition model}
\author{G\"oran H\"ogn\"as}
%
\institute{G\"oran H\"ogn\"as \at Centre of Excellence in Optimization and Systems Engineering and Department of Mathematics, \AA bo Akademi University, 20500 \AA bo, Finland, \email{ghognas@abo.fi}}
%
%
\maketitle

\abstract*{We model the evolution of two populations $U_t, V_t $ in competition by a two-dimensional size-dependent branching process. The population characteristics are assumed to be close to each other, as in a resident-mutant situation. 
Given  that $U_t = m$ and $V_t = n$ the expected values of $U_{t+1}$ and $V_{t+1}$ are given by $me^{r - K(m + bn)} $ and $ne^{\tilde r - \tilde K (n + am)}$,
respectively, where $r , \tilde r$ model the intrinsic population growth, $K, \tilde K $ model the force of inhibition on the population growth by the present population (such as scarcity of food), and $ a, b $ model the interaction between the two populations. For small $K, \tilde K
$ the process typically follows the corresponding deterministic Ricker competition model closely, for a very long time. 
Under some conditions, notably a mutual invasibility condition, the deterministic model has  a coexistence fixed point in the open first quadrant.  The asymptotic behaviour is studied through the
quasi-stationary distribution of the process. We initiate a study of those distributions as
the inhibitive force $K, \tilde K$ approach 0.}

\abstract{We model the evolution of two populations $U_t, V_t $ in competition by a two-dimensional size-dependent branching process. The population characteristics are assumed to be close to each other, as in a resident-mutant situation. 
Given  that $U_t = m$ and $V_t = n$ the expected values of $U_{t+1}$ and $V_{t+1}$ are given by $me^{r - K(m + bn)} $ and $ne^{\tilde r - \tilde K (n + am)}$,
respectively, where $r , \tilde r$ model the intrinsic population growth, $K, \tilde K $ model the force of inhibition on the population growth by the present population (such as scarcity of food), and $ a, b $ model the interaction between the two populations. For small $K, \tilde K
$ the process typically follows the corresponding deterministic Ricker competition model closely, for a very long time. 
Under some conditions, notably a mutual invasibility condition, the deterministic model has  a coexistence fixed point in the open first quadrant.  The asymptotic behaviour is studied through the
quasi-stationary distribution of the process. We initiate a study of those distributions as
the inhibitive force $K, \tilde K$ approach 0.
}

\section{Introduction}
\label{sec:1}

A general competition model between $k$ different species put forward in \cite{Schr} is
\begin{equation}
X^i_{t+1} = f_i(X_t, \xi_t)X_t^i, \enskip 1,2, \ldots k, \enskip t=0,1, \ldots
\end{equation}
where the state space {\bf S} is a closed subset of $\Re_+^k$ and the coordinate
axes form the extinction set $\bf S_0 $; the $\xi$'s represent a randomly evolving environment.

Since we concentrate on the two-dimensional case, $k = 2$, it will be more convenient to call  the components $X$ and $Y$, respectively. For a corresponding deterministic system we will use  lower case
$x$ and $y$ instead.

There are many different $f_i$ (growth factor of the $i$th species) proposed in the 
literature, e.\ g., the Leslie-Gower competition model (a generalized Beverton-Holt model)
which in its deterministic form may be written
\begin{eqnarray*} 
x_{t+1} = b_1 {1 \over {1 + x_t + c_1y_t}}x_t \\
y_{t+1} = b_2 {1 \over {1 + c_2x_t + y_t}}y_t .\\
\end{eqnarray*}
see \cite{Cush}. Here we consider a {\it Ricker-type} function, first in a deterministic form
\begin{eqnarray*} x_{t+1} = x_te^{r-x_t-by_t} \\
y_{t+1} = y_te^{\tilde r - ax_t - y_t},
\end{eqnarray*}
where $a$ and $b$ are taken close to 1, when the populations have largely the same characteristics, as in a resident-mutant situation. 

In the terminology of (1), with the randomness suppressed, we have here $f_1(x,y) = e^{r-x-by} $ and $f_2(x,y) = e^{\tilde r - ax - y} $. Denote
\begin{equation}\label{ricker}
F(x,y) = (xf_1(x,y), yf_2(x,y)) = (xe^{r-x-by}, ye^{\tilde r - ax - y}) 
\end{equation}
A very thorough study of the stability
properties of this map $F$, the {\it Ricker competition model} RCM, is provided by \cite{Luis}, see also \cite{Balr1}, \cite{Balr2}.

\bigskip
The one-dimensional Ricker model, introduced in \cite{Rick}
\begin{eqnarray*}
x_{t+1} = x_t e^{r - Kx_t}, \enskip t=0,1,2, \ldots 
\end{eqnarray*}
models the evolution of the population density of a species. 
\begin{itemize}

\item{$x_t$ is the density at time
$t$ or in the $t$th generation,}
\item{$r$ (or $e^r$) is the {\it birth rate, rate of natural increase,}}
\item{$K$ models the strength of an inhibitive factor, e.\ g., from internal competition
for scarce resources (in Ricker's original paper derived from predation of the adult upon the young),}
\item{the origin and the point $r \over K$ are the fixed points of the system.}
 \end{itemize}

\medskip
A key fact is: If $r $ and $x_0 $ are positive, the arithmetic averages of $x_t$ of the one-dimensional Ricker model {\it always} approach $r \over K$.

\medskip We note that if we pass to the normed process $Kx_t$ ($Kx$ being the new $x$) we can simplify the
one-dimensional Ricker equation to
\begin{equation}\label{rickernorm}
 x_{t+1} = x_t e^{r - x_t} .
 \end{equation}

\section{Demographic stochasticity}

In this section we introduce a two-dimensional competition model  which is essentially the
RCM as in the Introduction combined with {\it demographic stochasticity}.

Let $q$ and $\tilde q$ be two offspring distributions with means $e^r$ and $e^{\tilde r}$, respectively. We also assume that the have finite variances $\sigma^2$ and ${\tilde \sigma}^2$. 
Let $U_t$ and $V_t$ be two competing populations (modeled as branching processes) of the form
\begin{eqnarray}\label{dem} 
U_{t+1} = \sum_{j=1}^{U_t} \xi_{j,t}, \enskip t=0,1,2, \ldots \\
V_{t+1} = \sum_{j=1}^{V_t} \tilde \xi_{j,t} \enskip t =0,1,2, \ldots .
\end{eqnarray}
If $U_t = 0 $ ($V_t = 0$) then we define $U_{t+1} = 0$ ($V_{t+1} = 0$). In the definitions above the $\xi$'s and $\tilde \xi$'s are  independent nonnegative random variables with
distributions depending on the $t$th generation values:
\begin{eqnarray*}
{\bf P}\lbrace \xi_{1,t} = k \enskip | \enskip U_t = m, V_t = n \rbrace = \exp(-K(m+bn))q_k, \enskip k=1,2,3, \ldots \\
{\bf P}\lbrace \tilde \xi_{1,t} = k \enskip | \enskip U_t = m, V_t = n \rbrace = \exp(-\tilde K (am + n))\tilde q_k, \enskip k=1,2,3, \ldots
  \end{eqnarray*}
  when $m$ or $n$ are positive integers. $a$ models the relative 
influence of individuals of the $U$-species on the size of the litters $\tilde \xi$ of the $V$-species, and {\it vice
versa} for $b$.
\medskip
 
  The corresponding conditional probability of $\xi_{1,t} = 0 $ is given by
  the expression
  \begin{equation}\label{probzero}
  1 - \exp(-K(m+bn))(1-q_0)
  \end{equation}
  and similarly for $\tilde \xi_{1,t} $.
  
  \bigskip
$(U_t,V_t)$ is a Markov chain on the nonnegative quadrant of ${\bf Z}^2 $.  In principle, all values are attainable in one step from any interior point. Thus the interior of ${\bf Z}_+^2 $ constitutes
one communicating aperiodic irreducible class.

  \medskip
The conditional means of the normed processes $KU_{t+1}$ and $\tilde K V_{t+1} $ given
 $KU_t = x, \tilde K V_t=y $ are 
\begin{eqnarray}\label{means}
  x \exp (r-x-b{ K\over \tilde K}y), \quad
  y \exp (\tilde r - a{\tilde K \over K}x - y)
  \end{eqnarray}
  showing close analogy with the deterministic cases (\ref{ricker},\ref{rickernorm}).
  - Let us modify our definition of $F$ in (\ref{ricker}) slightly, replacing $a$ by
  $a{\tilde K \over K}$ and $b$ by $b {K \over \tilde K} $:
  \begin{equation}\label{rickermod}
  F(x,y) = (x \exp(r - x - b{ K \over \tilde K}y), y \exp(\tilde r - a{\tilde K \over K}x -y))
  \end{equation}
\medskip
  
  The conditional variance of $KU_{t+1}$ is
  \begin{equation}\label{variance}
  Kx\exp(-x - b {K \over \tilde K} y)\lbrace \sigma^2 + e^{2r}(1 - e^{-x-b{K\over \tilde K}y})\rbrace
  \end{equation}
  with a similar formula for $\tilde K V_{t+1} $. 
  
  \medskip
  Let us call the normed process $(X_t, Y_t)$; $X_t = KU_t$, $ Y_t = \tilde K V_t$. This is  the two-dimensional stochastic Ricker competition model to be investigated in this paper.

\medskip
For small values of $K$ and $\tilde K$ the $(X_t,Y_t)$ chain follows the path of
the deterministic process with very high probability: The formula (\ref{variance}) shows that the standard deviations are of  the order of $\sqrt K , \sqrt{\tilde K} $. Thus the process, whose both components are sums of independent random 
variables, is approximately normal with small standard deviation. More concretely,
suppose we start the normed process at $(x,y)$, then the next step is approximately normal
with mean given by (\ref{means}) and a standard deviation which can be made very small.
As shown in \cite{Kleb, qsd} we can infer the same thing for a finite number of steps,
$(X_t, Y_t) \approx F^t(X_0,Y_0)$ for $t=1,2, \ldots k $, when $K, \tilde K$ are small
enough. 
The difference is a 
sum of independent random variables and approximately bi-variate normal. 

We see from (\ref{probzero}) that the  extinction probability of the $X$-process in one step
given $X_t = x, Y_t = y$
is of the form
\begin{equation}
((1 - \exp(-x - b{K \over \tilde K}y)(1 - q_0))^{x \over K} \ge \delta^{1 \over K}
\end{equation}
where the positive number $\delta = 2^{-\log 2}$  is the minimum of the function 
$(1-e^{-x})^x $. 

Since the $Y$-component behaves similarly, we can infer that the extinction probability of our
bivariate stochastic process in one step is at least
\begin{equation}\label{delta}
\delta^{ {1 \over K} + {1 \over \tilde K}} 
\end{equation}
regardless of the present state of the process. Thus it has a finite life-time almost surely.

\section{Mutual invasibility}

The two-dimensional deterministic system $F$ defined by (\ref{rickermod}) has fixed points
at the origin (corresponding to the case when both species are extinct), at the point
$(r,0) $(the $y$-component is extinct and the $x$-component at its equilibrium), at $(0,\tilde r) $ and
possibly at
\begin{eqnarray}\label{fix}
(x^*, y^*) = ({{r - b{ K \over \tilde K}\tilde r} \over {1 - ab}}, {{\tilde r - a{\tilde K \over K}r} \over {1 - ab}} ). 
\end{eqnarray}    

We will assume conditions which guarantee the existence of the
fixed point  in the open first quadrant. $(x^*, y^*)$ is then called the {\it coexistence fixed point}.

Suppose $y_0 = \tilde r $ (seen, e.\ g., as the resident population density)) and $x_0$ (the mutant population density) is very small. Then the first component $x_1$  of $(x_1,y_1) = F(x_0,y_0)$ is larger than $x_0$ 
 if the condition $r > b {K \over \tilde K} \tilde r $ is met. The interpretation is that the mutant $x$ can increase from small values while the resident population is at its equilibrium. We say that $x$ can {\it invade} $y$. Conversely, $y$ can invade $x$ if
 $\tilde r > a {\tilde K \over K} r $.
 
 We will impose two conditions: \begin{equation}\label{ab} ab < 1
\end{equation} and  the {\it mutual invasibility condition}
\begin{equation}\label{mutual}
r > b {K \over \tilde K} \tilde r ,\quad
\tilde r > a {\tilde K \over K} r .
\end{equation}

It is immediately seen that this is sufficient to ensure the existence of a coexistence fixed point $(x^*, y^*)$ in the open first quadrant. The fixed point may be repelling or 
attracting; the conditions can be found in \cite{Luis}. 
\medskip

Remark.
If the population characteristics are very close to each other one could well choose $a=b=1$. In \cite{Fager} we investigated such a case: there is no coexistence fixed point so one of
the two populations will eventually take over completely.

\medskip

Without regard to the attractivity properties of the fixed points, the conditions (\ref{ab},\ref{mutual}) guarantee the existence of an invariant compact set $C$ inside the open first
quadrant: $F(C) \subset C $. For the detailed calculations, see \cite{invar}.

In \cite{invar} it is also proved that there exists a compact set $C_1$ such that it is
mapped "far inside" $C_1$ by some iterate $F^N$ of $F$. To be more precise, there is a compact set $C_1$ inside the open
first quadrant such that the Euclidean distance between the outside of $C_1$ and $F^N(C_1)$ is
strictly positive:
\begin{equation}\label{dist}
 d(C_1^c, F^N(C_1)) > 0 .
\end{equation}
 Another way of putting it is to say that $F^N(C_1)$ is farther away from the coordinate axes than $C_1$.
 
 \section{The quasi-stationary distribution}
 
 The analysis of the process $(U_t, V_t)$ and its normed version $(X_t, Y_t)$ follows 
 the outline in \cite{qsd}.
 
 The Markov chain $(U_t, V_t)$ in the interior of the first quadrant $\bf Z_+^2$ was found above to have a finite lifetime a.s., see (\ref{delta}). For any $(m,n)$, both components
positive, the probability of moving straight to the origin is  the product of expressions such as (\ref{probzero}):
$$    (1 - \exp(-K(m+bn))(1-q_0))^m (1 - \exp(-\tilde K(am+n))(1-\tilde q_0))^n $$
which can be minorized by an expression of the form
\begin{equation}
(1 - \exp(-L(m+n))(1-\bar q_0))^{m+n} \approx 1 - (m+n) \exp(-L(m+n))(1-\bar q_0)
\end{equation} 
where $L$ is taken such that $L(m+n) \le \min(K(m+bn), \tilde K(am+n))$ and $\bar q_0 =
\max(q_0, \tilde q_0)$.

Following the corresponding proof in \cite{qsd} we can now draw the conclusion that the entries of the transition probability matrix $Q$ of the Markov chain restricted to the strictly positive
states $(m,n)$ have a finite sum. The sub-markovian transition matrix has a spectral radius (the Krein-Rutman maximal eigenvalue $\lambda$) bounded above by an expression of the form $$ 1 - \delta^{{1 \over K }+ {1 \over \tilde K} }$$ where $\delta $ again is the strictly positive minimum of the function $(1 - e^{-x})^x $. This is a rough upper bound because the
complement of the all-positive states is the union of the non-negative coordinate  axes, and not just the origin $(0,0)$.

The Markov chain restricted to the strictly positive states admits a {\it quasi-stationary distribution} (qsd) $\pi$. One of the characterizations of the qsd is that it is the left 
eigenvector corresponding to the dominant eigenvalue $\lambda$: $$ \lambda \pi =  \pi Q .$$
Another is the following invariance property: If the initial distribution is $\pi$ then the conditional distribution of the first step given that it is all-positive is $\pi$, too:
\begin{eqnarray*}
{\bf P}_{\pi} \lbrace (U_1,V_1) = (\cdot, \cdot) \enskip| \enskip U_1 > 0, V_1 > 0 \rbrace = \pi(\cdot, \cdot) 
\end{eqnarray*}
where we use the notation ${\bf P}_{\pi} $ to indicate that the initial distribution
(the distribution of $(U_0, V_0)$) is $\pi$.

\section{Asymptotics}

Our asymptotic analysis follows the reasoning in \cite{qsd} closely. The quasi-stationary
distribution $\pi$ depends on $K$ and $\tilde K$. (We assume $a$ and $b$ to be constant.) As $K$ and $\tilde K$ go to 0 the family
of probability measures $\pi_{K, \tilde K} $ is {\it tight} provided $K$ and $\tilde K$ vary boundedly in such a way that the fixed point $(x^*, y^*)$ in (\ref{fix}) stays in a compact
set within the interior of the first quadrant. For simplicity we take $K = \tilde K $, $0 < K < 1$, from now on. 

\begin{theorem}\label{tight} Let $K = \tilde K$. Let (\ref{ab},\ref{mutual}) be satisfied. Consider the Markov chains $(X_t, Y_t)$ for $K > 0 $ and their quasi-stationary distributions $\pi_K$.
The family of probability measures $ \pi_K, \enskip K > 0 $, is tight. 
\end{theorem} 

\begin{proof} Recall that the conditional mean of $(X_{t+1},Y_{t+1})$ given that $(X_t, Y_t)=
(x,y) $ is $F(x,y) \equiv (xf_1(x,y), yf_2(x,y))$, (\ref{ricker}), and the variance $v_x$
of its first component is
  \begin{equation}\label{variancemod}
  Kx\exp(-x - b y)\lbrace \sigma^2 + e^{2r}(1 - e^{-x-by})\rbrace
  \end{equation}with a similar formula for the second component, see (\ref{variance}).  

Following \cite{qsd}, Section 3, let $S(s)$ be the moment generating function (mgf) of the offspring distribution $q$ and $c_{(x,y)}(s) $  
the mgf of the variables $\xi_{j,t}$ of (\ref{dem}) given that $X(t)=x$ and $Y(t)= y$.
The entropy function $c_{(x,y)}^*(z)$ of  $\xi_{j,t}$ is $\sup_{s \in \Re}[zs - c_{(x,y)}(s)] $. As in \cite{qsd}, pp.\ 252-253, we get, for any positive $a$,
$$ {\bf P}\lbrace X_{t+1} > a + xf_1(x,y) \enskip | \enskip X_t = x, Y_t = y \rbrace
\le \exp( - s ({a \over K} - {x \over K}f_1(x,y)) + {x \over K} c_{(x,y)}(s)) $$
which is at most
$$ \exp(- {x \over K}c_{(x,y)}^*({a \over x} + e^{r-x-by})) .$$
The entropy function is 0 at the conditional mean $e^{r-x-by}$ of $\xi_{j,t}$ and at least
${1 \over {3v_x}}({a \over x})^2 $ for small $a$.

Since the variance varies continuously with $x$ and $y$ we can conclude that the exponent
can be bounded away from 0 when $(x,y)$ lies in a compact set in the interior of the first
quadrant.

Thus, given a small positive $a$,  there exists a $w > 0$ such that
\begin{equation} \label{expbound}
{\bf P} \lbrace X_{t+1} > xf_1(x,y) + a \enskip | \enskip X_t=x, Y_t=y \rbrace \le 
\exp(-{w \over K}) 
\end{equation}
(and similarly for $Y_{t+1} $) for all $(x,y)$ in a given compact set in the open first quadrant.

Putting it in another way, we have 
\begin{equation} \label{expboundnbhd}
{\bf P} \lbrace (X_{t+1},Y_{t+1}) \in U(F(x,y)) \enskip | \enskip (X_t, Y_t) = (x,y) \rbrace \ge 
1 - \exp(-{w \over K}) 
\end{equation}
where $U(F(x,y))$ is a neighbourhood of the point $F(x,y)$.

If we take as our compact set an invariant set $C$, such as the one constructed in \cite{invar}, and as $U(C)$ an open set containing $C$ we
get
\begin{equation} \label{expboundU}
{\bf P} \lbrace (X_{t+1},Y_{t+1}) \in U(C) \enskip | \enskip (X_t, Y_t) \in C \rbrace \ge 
1 - \exp(-{w \over K}) .
\end{equation}

Since $F$ is continuous we can extend the above relation, by appropriately adjusting the
quantity $w > 0$, to the $N$th iterate of $F$:
\begin{equation} \label{expbounditer}
{\bf P} \lbrace (X_{t+N},Y_{t+N}) \in U(C) \enskip | \enskip (X_t, Y_t) \in C \rbrace \ge 
1 - \exp(-{w \over K}) .
\end{equation}

Now we let $N$ and $C_1$ be as in \cite{invar}, i.\ e., the interior of $C_1$ is an open 
set containing $F^N(C_1)$. Then we obtain
\begin{equation} \label{expboundN}
{\bf P} \lbrace (X_N,Y_N) \in C_1 \enskip | \enskip (X_0, Y_0) \in C_1 \rbrace \ge 
1 - \exp(-{w \over K}) 
\end{equation}
from which we can draw the conclusion, as in \cite{qsd}, that the dominant eigenvalue $\lambda = \lambda (K)$ of
the transition probability matrix of the Markov chain $(X_t,Y_t)$ restricted to the
positive quadrant of $\Re^2$, differs "exponentially little" from 1:
\begin{equation}\label{lambda}
1 - \exp(-{ u \over K}) < \lambda (K)  < 1 
\end{equation}
 for some $u > 0 $.

\medskip
Starting the Markov chain from the point $(X_0, Y_0) = (x,y)$ the distribution of $(X_1, Y_1)$
has mean $F(x,y)$ and a uniformly bounded variance, see (\ref{variancemod}). Hence there is
a compact set $B \subset \Re^2$ such that the
probability of $(X_1, Y_1)$ being outside of $B$ is $ < \varepsilon$ for all $(x,y)$ and all $K$. As in \cite{qsd} we conclude that $\pi_K (B^c) < 2 \varepsilon$.

\medskip
In a strip $S$ close to the $y$-axis, the function $(1 - e^{-x})^x $ approaches  1, it is $ > e^{-v} $ for some small positive $v$. Take $v$ so small that $0< v < u$ where $u$ satisfies (\ref{lambda}) above. The probability
of immediate exit from the state space (the open first quadrant) is at least $e^{-{v \over K}}$ starting from a point in $S$. If  the initial point is chosen according to law $\pi_K$ 
then the exit probability is $1 - \lambda (K) $ (by the definition of qsd). As in \cite{qsd}
we use (\ref{lambda}) to conclude that $$\pi_K(S)  < \exp(-{{u-v}\over K}) $$ which can be made $ < \varepsilon$
for $K$ small enough. A similar reasoning can be applied to a strip above the $x$-axis.

\medskip
The above remarks show that, for given positive $\varepsilon$, there exists a compact set inside the open first quadrant whose
$\pi_K$-measure is $> 1-4\varepsilon $ for all $K$ small enough.  

\medskip
Thus the family $\pi_K, \enskip K>0 $, is tight.

\end{proof}

\begin{theorem}\label{attr} The limit measures of $\pi_K$ as $K \to 0$ are invariant with respect to the
deterministic system $F$. If the coexistence fixed point $(x^*, y^*)$ of (\ref{fix}) is
attracting then the quasi-stationary distributions converge to the point mass at the fixed point: $$\lim_{K \to 0} \pi_K  = \delta_{(x^*,y^*)} .$$
\end{theorem}

 \begin{proof} (Sketch) The invariance of the limit measures is proved as in Step 5 of
 \cite{qsd}, pp.\ 254f. We know, see \cite{Luis}, that the coexistence fixed point
 $(x^*, y^*)$ is attracting if $2(1-b)\tilde r + 2(1-a)r - 4(1-ab) \le (r-b\tilde r)
 (\tilde r - ar) < (1-b)\tilde r + (1-a) r $. Also, by \cite{Balr2}, if this 
condition is satisfied it is globally attracting so that the point mass at $(x^*, y^*)$ is the unique invariant measure for the deterministic system $F$ 
 inside the open first quadrant of $\Re^2$.

\end{proof}

\section{Conclusion and outlook}

Theorems \ref{tight} and \ref{attr} show that the branching process $(U_t, V_t)$ and its
normed version $(X_t, Y_t)$, even if it eventually goes extinct, emulates the deterministic process for a long time. The dominant eigenvalue is within $e^{-{u\over K}} $ of 1 so the
expected lifetime $$ {1 \over {1-\lambda(K)} }$$ of the process is exponential in ${1 \over K} $. When the coexistence fixed point $(x^*, y^*)$
is attracting we can say more: the quasi-stationary 
distribution is a "blurred version" of the point mass at $(x^*, y^*)$. The linear autoregressive
approximation of $(X_t, Y_t)$ in the vicinity of the fixed point, which was introduced by \cite{Kleb} in a similar one-dimensional case, can be used as in \cite{Jung} to obtain 
rough estimates of the quantity $u$ in (\ref{lambda}). This in turn gives us an estimate for the expected life-time.

If the coexistence fixed point is repelling, the picture is less clear. We know by
Theorem \ref{tight} that the qsd's do have limits but their precise nature is unknown. 
Simulations suggests a conjecture for $r, \tilde r$ not too large: the limits will be
supported by the attracting cycles of the deterministic system $F$ 
(as in the one-dimensional Ricker case, see \cite{qsd}). A successful completion of this program is likely to require a thorough understanding of the two-dimensional dynamical system generated by $F$.

\medskip
\begin{acknowledgement}
The author has worked on Ricker competition models with Henrik Fagerholm, Mats Gyllenberg, and Brita Jung. The present paper was presented at the International Conference on Difference Equations and Applications ICDEA 2012 in Barcelona. Many discussions took place in the constructive and inspiring conference atmosphere. In particular, questions and comments by  Rafael Lu\'{i}s,
Eduardo Cabral Balreira, and Saber Elaydi are gratefully acknowledged. Travel funding was provided by the Magnus Ehrnrooth Foundation.
\end{acknowledgement}
\end{document}